\documentclass[a4paper,reqno,11pt]{amsart}

\usepackage{amsmath, amsfonts, amssymb, amsthm, amscd}
\usepackage{graphicx}
\usepackage{psfrag}
\usepackage{perpage}
\usepackage{url}
\usepackage{color}

\usepackage[utf8]{inputenc}
\usepackage[T1]{fontenc}
\usepackage{microtype}

\usepackage[a4paper,scale={0.72,0.74},marginratio={1:1},footskip=7mm,headsep=10mm]{geometry}

\usepackage{hyperref}

\numberwithin{equation}{section}

\newtheorem{theorem}{Theorem}[section]
\newtheorem{lemma}[theorem]{Lemma}

\newtheorem{remark}[theorem]{Remark}

\newtheorem{assumption}[theorem]{Assumption}

\newcommand{\bbE}{{\ensuremath{\mathbb E}} }

\newcommand{\bbP}{{\ensuremath{\mathbb P}} }

\newcommand{\cA}{{\ensuremath{\mathcal A}} }

\newcommand{\cC}{{\ensuremath{\mathcal C}} }

\newcommand{\cF}{{\ensuremath{\mathcal F}} }
\newcommand{\cG}{{\ensuremath{\mathcal G}} }

\DeclareMathSymbol{\leqslant}{\mathalpha}{AMSa}{"36} 
\DeclareMathSymbol{\geqslant}{\mathalpha}{AMSa}{"3E} 
\DeclareMathSymbol{\eset}{\mathalpha}{AMSb}{"3F}     
\newcommand{\dd}{\text{\rm d}}             

\newcommand{\R}{\mathbb{R}}

\newcommand{\Z}{\mathbb{Z}}
\newcommand{\N}{\mathbb{N}}

\newcommand{\PEfont}{\mathrm}
\DeclareMathOperator{\var}{\ensuremath{\PEfont Var}}

\DeclareMathOperator{\p}{\ensuremath{\PEfont P}}
\DeclareMathOperator{\e}{\ensuremath{\PEfont E}}

\renewcommand\P{\PEfont P}

\DeclareMathOperator{\bbvar}{\ensuremath{\mathbb{V}ar}}

\newcommand{\ind}{{\sf 1}}

\renewcommand{\epsilon}{\varepsilon} 
\renewcommand{\theta}{\vartheta} 
\renewcommand{\rho}{\varrho} 
\renewcommand{\phi}{\varphi}

\newenvironment{myenumerate}{%
\renewcommand{\theenumi}{\arabic{enumi}}%
\renewcommand{\labelenumi}{{\rm(\theenumi)}}%
\begin{list}{\labelenumi}
	{%
	\setlength{\itemsep}{0.4em}%
	\setlength{\topsep}{0.5em}%
	\setlength\leftmargin{2.45em}%
	\setlength\labelwidth{2.05em}%
	\setlength{\labelsep}{0.4em}%
	\usecounter{enumi}%
	}%
	}%
{\end{list}
}

\renewenvironment{enumerate}{
\begin{myenumerate}}%
{\end{myenumerate}}

\newenvironment{myitemize}{%
\begin{list}{$\bullet$}%
 	{%
	\setlength{\itemsep}{0.4em}%
	\setlength{\topsep}{0.5em}%
	\setlength\leftmargin{2.45em}%
	\setlength\labelwidth{2.05em}%
	\setlength{\labelsep}{0.4em}%
	}%
	}%
{\end{list}}

\renewenvironment{itemize}{
\begin{myitemize}}%
{\end{myitemize}}

\MakePerPage[2]{footnote} 

\title[Smoothing inequality for disordered polymers]{A general smoothing inequality\\
for disordered polymers}

\author{Francesco Caravenna}
\address{Dipartimento di Matematica e Applicazioni\\
 Universit\`a degli Studi di Milano-Bicocca\\
 via Cozzi 53, 20125 Milano, Italy}
\email{francesco.caravenna@unimib.it}

\author{Frank den Hollander}
\address{Mathematisch instituut\\
Universiteit Leiden\\
Postbus 9512\\
2300 RA Leiden\\
The Netherlands}
\email{denholla@math.leidenuniv.nl}


\date{\today}

\def\dd{\mathrm{d}}

\def\f{\textsc{f}}
\def\M{\mathrm{M}}

\begin{document}

\begin{abstract}
This note sharpens the smoothing inequality of Giacomin and Toninelli
\cite{cf:GT1}, \cite{cf:GT2} for disordered polymers. This inequality is 
shown to be valid for any disorder distribution with locally finite exponential 
moments, and to provide an asymptotically sharp constant for weak disorder. 
A key tool in the proof is an estimate that compares the effect on the free 
energy of tilting, respectively, shifting the disorder distribution. This estimate 
holds in large generality (way beyond disordered polymers) and is of 
independent interest.
\end{abstract}

\maketitle


\section{Introduction and main results}

Understanding the effect of disorder on phase transitions is a key topic in 
statistical physics. In a celebrated paper, Harris~\cite{cf:Harris} proposed 
a criterion that predicts whether or not the addition of an arbitrarily small 
amount of quenched disorder is able to modify the critical behavior of a 
system close to a phase transition. The rigorous justification of this criterion 
for a class of \emph{pinning models} has been an active direction of research 
in the mathematical literature (see Giacomin~\cite{cf:Gia2} for an overview). 
One of the key tools in this program is the \emph{smoothing inequality} of 
Giacomin and Toninelli~\cite{cf:GT1}, \cite{cf:GT2}. It is the purpose of this 
note to generalize and sharpen this inequality.

Section~\ref{sec:motivation} provides motivation, Section~\ref{sec:assumptions}
states the necessary model assumptions, Section~\ref{sec:free energy} defines
the free energy, Section~\ref{sec:theorems} states our main theorems, while 
Section~\ref{sec:discussion} discusses the context of these theorems. Proofs
are given in Sections~\ref{sec:protilt}--\ref{sec:smoothshift}. 


\subsection{Motivation}
\label{sec:motivation}

We begin by describing a class of models that motivates our main results
in Section~\ref{sec:theorems}. We use the
notation $\N := \{1,2,\ldots\}$ and $\N_0 := \N \cup \{0\}$.

Consider a recurrent Markov chain $S := (S_n)_{n\in\N_0}$ on a countable set 
$\mathsf{E}$, starting at a distinguished point denoted by $0$, defined 
on a probability space $(\Omega,\cF,\p)$, and let $\tau_1 := \inf\{n \in \N
\colon\, \ S_n = 0\}$ be its first return time to $0$. The key assumption is 
that for some $\alpha \in [0,\infty)$, 
\begin{equation} 
\label{eq:astau}
\p(\tau_1 > n) = n^{-\alpha + o(1)}, \qquad n \to \infty.
\end{equation}
The case of a transient Markov chain, i.e., $\p(\tau_1 = \infty) > 0$, can be 
included as well, and requires that \eqref{eq:astau} holds conditionally on 
$\{\tau_1<\infty\}$.

Given an $\R$-valued sequence $\omega := (\omega_n)_{n\in\N}$ (the 
\emph{disorder} sequence), a function $\phi\colon\,\mathsf{E} \to \R$ 
(the \emph{potential}), and parameters $N \in\N$, $\beta \geq 0$, 
$h\in\R$ (the \emph{system size}, the \emph{disorder strength} and 
the \emph{disorder shift}), we define the \emph{partition function}
\begin{equation} 
\label{eq:model}
Z_{N, \beta, h}^{\omega,\phi} \:=
\e\Big[ e^{\sum_{n=1}^N (h + \beta\omega_n) \phi(S_n)} \, 
\ind_{\{S_N = 0\}} \Big] \,\in\, [0,\infty],
\end{equation}
i.e., at each time $n$ the Markov chain gets an exponential reward or penalty 
proportional to $h + \beta \omega_n$, modulated by a factor $\phi(S_n)$. The 
sequence $\omega$ is to be thought of as a typical realization of a random 
process. Note that
\begin{itemize}
\item 
the choice $\phi(x) := \ind_{\{0\}}(x)$ corresponds to the \emph{pinning model} 
(see Giacomin~\cite{cf:Gia}, \cite{cf:Gia2}, den Hollander~\cite{cf:dH});
\item 
when $\mathsf{E} = \Z$ and $S$ is nearest-neighbor with symmetric excursions
out of $0$, the choice $\phi(x) := \ind_{(-\infty,0]}(x)$ corresponds to the 
\emph{copolymer model} (see \cite{cf:Gia}, \cite{cf:dH});
\footnote{The standard copolymer model is defined through a \emph{bond} 
interaction: $\phi(S_n)$ is replaced by $\phi(S_{n-1},S_n) := \ind_{(-\infty, 0]}
(\frac{1}{2}[S_{n-1} + S_n])$, and $(\beta,h)$ by $(-2\lambda,-2\lambda h)$. 
This can be still cast in the framework of \eqref{eq:model} by picking 
$\mathsf{E} = \Z^2$, taking the pair process $(S_{n-1},S_n)$ as the Markov 
chain, and $(0,0)$ as $0$.}
\end{itemize}
Thus, the modulating potential $\phi$ allows us to interpolate between different 
classes of models. When $S$ is simple random walk on $\Z^d$ and $\phi(x) 
\approx |x|^{-\theta}$ as $|x| \to\infty$ for some $\theta \in (0,\infty)$, the model 
displays interesting features that are currently under investigation (Caravenna
and den Hollander~\cite{cf:CardH}).


\subsection{Assumptions}
\label{sec:assumptions}

Although our main focus will be on the model in \eqref{eq:model}, we list the 
assumptions that we actually need. We start with the disorder.

\begin{assumption}[The disorder]
\label{ass:disorder}
The disorder $\omega = (\omega_n)_{n\in\N}$ is an i.i.d.\ sequence of $\R$-valued 
random variables, defined on a probability space $(\Omega', \cF', \bbP)$, such that
\begin{gather}
\label{eq:t0}
\exists\,\, t_0 \in (0, \infty]\colon\,
\quad \M(t)  := \bbE\big[ e^{t\omega_1} \big] < \infty \quad \forall\,\, |t| < t_0, \\ \nonumber
\bbE[\omega_n] = 0, \quad \bbvar(\omega_n) = 1.
\end{gather}
\end{assumption}

\noindent
The crucial assumption is that the disorder distribution has locally finite exponential 
moments. The choice of zero mean and unit variance is a convenient normalization 
only (since we can play with the parameters $\beta$ and $h$).

For $\delta \in (-t_0, t_0)$, we denote by $\bbP_\delta$ the \emph{tilted law} under 
which $\omega = (\omega_n)_{n\in\N}$ is i.i.d.\ with marginal distribution 
\begin{equation}
\label{eq:RNn}
\bbP_\delta(\omega_1 \in \dd x) := e^{\delta x - \log\M(\delta)}\,\bbP(\omega_1 \in \dd x).
\end{equation} 

Next we state our assumptions on the partition function $Z_{N,\omega,\beta,h}$ we will be
able to handle, defined for $N \in \N$, $\beta \geq 0$, $h\in\R$ and $\bbP$-a.e. $\omega \in 
\R^\N$ (keeping in mind \eqref{eq:model} as a special case).

\begin{assumption}[The partition function {[I]}] 
\label{ass:model}
$Z_{N,\omega,\beta,h}$ is a
measurable function defined on $\N \times \R^\N \times [0,\infty) \times \R$,
taking values in $[0,\infty)$ and satisfying the following conditions:
\begin{enumerate}
\item
\label{it:basic} 
$Z_{N,\omega,\beta,h}$ is a function of $N$ and of $(h + \beta \omega_n)_{1 \leq n \leq N}$.
\item
\label{it:superadd} 
$Z_{N+M, \omega, \beta, h} \geq Z_{N, \omega, \beta, h} \, Z_{M, \theta^N\omega, \beta, h}$
for all $N,M \in \N$, where $\theta$ is the left-shift acting on $\omega$, i.e., $(\theta^N \omega)_n 
:= \omega_{N+n}$ for $N\in\N$.
\item
\label{it:polylb} 
There exists a $\gamma \in (0,\infty)$ such that, for $N$ in a subsequence of $\N$,
\begin{equation}
\label{eq:polylb}
Z_{N,\omega,\beta,h} \geq \frac{c_{\beta,h}(\omega)}{N^\gamma} 
\quad \text{with} \quad
\bbE_\delta[\log c_{\beta,h}(\omega)] > -\infty \quad \forall\,\,\delta \in (-t_0, t_0).
\end{equation}
\end{enumerate}
\end{assumption}

\begin{remark}\rm \label{rem:pincop}
Note that properties \eqref{it:basic} and \eqref{it:superadd} are satisfied for the model 
in \eqref{eq:model}. For property \eqref{it:polylb} to be satisfied as well, we need 
to make additional assumptions on $\phi$ and/or $S$. For instance, for the pinning 
model 
property \eqref{it:polylb} holds with $\gamma = (1+\alpha)+\epsilon$, for 
any fixed $\epsilon > 0$ (and for a suitable choice of $c_{\beta,h}(\omega) 
= c^\epsilon_{\beta,h}(\omega)$), which follows from \eqref{eq:astau} after restricting 
the expectation in \eqref{eq:model} to the event $\{\tau_1 = N\}$. Alternatively, when 
$\mathsf{E} = \Z^d$, if $\phi$ vanishes in a half-space and $S$ is symmetric (as for 
the copolymer model), property \eqref{eq:polylb} with $\gamma = (1+\alpha) +\epsilon$ 
again follows from \eqref{eq:astau}.
\end{remark}

As a matter of fact, properties \eqref{it:basic} and \eqref{it:superadd} are rather mild: they 
are satisfied for many $(1+d)$-dimensional directed models (possibly after a minor 
modification of the partition function that does not change the free energy defined 
below). In contrast, property \eqref{it:polylb} is a more severe restriction. Roughly 
speaking, it says that the \emph{disorder can be avoided at a cost that is only 
polynomial in the system size}.


\subsection{Free energy}
\label{sec:free energy}

If Assumptions~\ref{ass:disorder} and~\ref{ass:model} are satisfied, then we can 
define the \emph{free energy} 
\begin{equation} 
\label{eq:freeen}
\f(\beta,h; \delta) := \limsup_{N\to\infty} \frac{1}{N} 
\bbE_\delta \big[ \log Z_{N, \omega, \beta, h} \big]
\end{equation}
for $\beta \geq 0$, $h \in \R$, $\delta \in (-t_0, t_0)$ when $\omega$ is chosen 
according to $\bbP_\delta$. 

\begin{remark}\rm\label{rem:kingman}
(a) In the general framework of Assumption~\ref{ass:model}, it may happen that 
$\f(\beta,h; \delta) = \infty$ for some values of the parameters. However, for the 
model in \eqref{eq:model} we have $\f(\beta,h; \delta) < \infty$ as soon as $\phi$ 
is bounded (see \eqref{eq:boundham} below).\\
(b) By the super-additivity property \eqref{it:superadd} in Assumption~\ref{ass:model}, the $\limsup$ in 
\eqref{eq:freeen} may be replaced by $\sup$, or by $\lim$ restricted to those values 
of $N$ for which $\bbE_\delta \big[ \log Z_{N, \omega, \beta, h} \big] > -\infty$, which 
by properties \eqref{it:superadd}--\eqref{it:polylb} form a sub-lattice $\textsc{t}\N$. By 
Kingman's super-additive ergodic theorem, we may also remove the expectation 
$\bbE_\delta$ in \eqref{eq:freeen}, because the limit as $N\to\infty$, $N \in \textsc{t}\N$, 
exists and is constant $\bbP_\delta$-a.s.
\end{remark}

A direct consequence of \eqref{eq:polylb} is the inequality $\f(\beta,h; \delta) \geq 0$, 
which is a crucial feature of the class of models we consider. In many interesting 
cases, like for pinning and copolymer models, the free energy is zero in some 
closed region of the parameter space and strictly positive in its complement, with
both regions non-empty. When this happens, the free energy is not an analytic 
function and the model is said to undergo a \emph{phase transition}. It is then of 
physical and mathematical interest to study the regularity of the free energy close 
to the \emph{critical curve} separating the two regions.

More concretely, consider the case when $h \mapsto Z_{N,\omega,\beta,h}$ is 
monotone (like for the model in \eqref{eq:model} when $\phi$ has a sign), say
non-decreasing, so that $h \mapsto \f(\beta,h;\delta)$ is non-decreasing as well. 
Then for every $\beta \geq 0$ there exists a critical value $h_c(\beta) \in \R \cup 
\{\pm\infty\}$ such that $\f(\beta,h;0) = 0$ for $h < h_c(\beta)$ and $\f(\beta,h;0) 
> 0$ for $h > h_c(\beta)$ (we consider $\delta = 0$ for simplicity). If $h \mapsto 
\f(\beta,h;0)$ is continuous as well, as is typical, then $\f(\beta,h_c(\beta);0) = 0$ 
and it is interesting to understand how the free energy vanishes as $h \downarrow 
h_c(\beta)$. For homogeneous pinning models, i.e., when $\beta = 0$, it is known 
that
\begin{equation} 
\label{eq:hompin}
\f(0, h_c(0) + t;0) =  t^{\max\{\frac{1}{\alpha},1\} + o(1)}, \qquad t \downarrow 0.
\end{equation}
(See \cite[Theorem~2.1]{cf:Gia} for more precise estimates.) On the other hand,
as soon as disorder is present, i.e., when $\beta > 0$, it was shown by Giacomin 
and Toninelli~\cite{cf:GT1}, \cite{cf:GT2} that, under some mild restrictions on the 
disorder distribution,
\begin{equation} 
\label{eq:smoo0}
\exists\,c \in (0,\infty)\colon\quad 0 \leq \f(\beta,h_c(\beta) + t;0) 
\leq \frac{c}{\beta^2} \, t^2.
\end{equation}
Comparing \eqref{eq:hompin} and \eqref{eq:smoo0}, we see that when $\alpha 
> \frac{1}{2}$ the addition of disorder has a \emph{smoothing} effect on the way 
in which the free energy vanishes at the critical line.


\subsection{Main results}
\label{sec:theorems}

The goal of this note is to generalize and sharpen \eqref{eq:smoo0}, namely, to 
show that no assumption on the disorder distribution other than \eqref{eq:t0} is 
required, and to provide estimates on the constant $c$ that are optimal in some 
sense (see below). We will stay in the general framework of Assumption~\ref{ass:model}, 
with no mention of ``critical lines''. 


\subsubsection{$\bullet$ Tilting}

First we prove a smoothing inequality for $\f(\beta,h;\delta)$ with respect 
to the tilt parameter $\delta$ rather than the shift parameter $h$. Although both 
tilting and shifting are natural ways to control the disorder bias, the latter is often 
preferred in the literature because the free energy typically is a convex function of 
the shift parameter $h$ (like for the model in \eqref{eq:model}). However, for the 
purpose of the smoothing inequality the tilt parameter $\delta$ turns out to be more 
natural.

\begin{theorem}[Smoothing inequality with respect to a disorder tilt]
\label{th:smoothing}
Subject to Assumptions~{\rm \ref{ass:disorder}} and {\rm \ref{ass:model}}, if 
$\f(\bar\beta, \bar h; 0) = 0$ for some $\bar\beta > 0$ and $\bar h \in \R$, then 
for all $\delta \in (-t_0, t_0)$,
\begin{equation} 
\label{eq:smodelta}
0 \leq \f(\bar\beta, \bar h; \delta) \leq \frac{\gamma}{2} \, B_\delta \, \delta^2
\end{equation}
where the constants $t_0$ and 
$\gamma$ are defined in \eqref{eq:t0} and \eqref{eq:polylb}, while
\begin{equation} 
B_\delta := \frac{2}{\delta} \bigg| (\log\M)'(\delta) - \frac{\log\M(\delta)}{\delta} \bigg|
\in (0,\infty)
\qquad \text{satisfies} \qquad \lim_{\delta \to 0} B_\delta = 1 .
\end{equation}
\end{theorem}

\begin{remark}\rm
For pinning and copolymer models satisfying \eqref{eq:astau}, we can set $\gamma = 1+\alpha$ 
in \eqref{eq:smodelta}, by Remark~\ref{rem:pincop}.
\end{remark}

\noindent
Theorem~\ref{th:smoothing} is proved in Section~\ref{sec:protilt} through a direct translation of 
the argument developed in Giacomin and Toninelli~\cite{cf:GT2}. The proof is based on the 
concept of \emph{rare stretch strategy}, which has been a crucial tool in the study of disordered 
polymer models since the papers by Monthus~\cite{cf:Monthus}, Bodineau and 
Giacomin~\cite{cf:BodGia}.


\subsubsection{$\bullet$ Shifting}

Next we consider the effect of a disorder shift. In the \emph{Gaussian case}, i.e., when 
$\bbP(\omega_1 \in \cdot)=N(0,1)$, tilting is the same as shifting: in fact
$\bbP_\delta (\omega_1 \in 
\cdot)= N(\delta,1)$ and so $\omega_n$ under $\bbP_\delta$ is distributed like $\omega_n 
+ \delta$ under $\bbP$. Recalling property \eqref{it:basic}, 
we then get
\begin{equation} 
\label{eq:tiltshift}
\f(\beta, h; \delta) = \f(\beta, h + \beta \delta; 0)
\end{equation}
and, since $\M(\delta) = e^{\delta^2 / 2}$, it follows from \eqref{eq:smodelta} that if $\f(\bar\beta, 
\bar h; 0) = 0$ with $\bar\beta > 0$, then
\begin{equation}
0 \leq \f(\bar\beta, \bar h + t; 0) \leq \frac{\gamma}{2 \bar\beta^2} \, t^2 
\qquad \forall\, t \in \R.
\end{equation}
This is precisely the smoothing inequality with respect to a disorder shift in \eqref{eq:smoo0}, 
with an explicit constant 
(see also Giacomin~\cite[Theorem 5.6 and Remark 5.7]{cf:Gia}).

For a general disorder distribution tilting is different from shifting. However, we may still hope 
that \eqref{eq:tiltshift} holds approximately. This is what was shown in Giacomin and 
Toninelli~\cite{cf:GT1}, under additional restrictions on the disorder distribution and with 
non-optimal constants. 
The main result of this note, Theorem~\ref{th:compdeltah} below, shows that the effects
on the free energy of tilting or shifting the disorder distribution are asymptotically equivalent, 
in large generality and with asymptotically optimal constants in the weak interaction limit. 
Since this result is unrelated to Theorem~\ref{th:smoothing} and is of independent interest, 
we formulate it for a very general class of statistical physics models, way beyond disordered 
polymer models.

\begin{assumption}[The partition function {[II]}] 
\label{ass:model2}
The partition function is defined as
\begin{equation} 
\label{eq:ZNgen}
Z_{N, \omega, \beta, h} := \e_N \Big[ e^{\sum_{n=1}^N (h + \beta \omega_n) \sigma_n}  \Big],
\end{equation}
where, for fixed $N\in\N$, $(\sigma_i)_{1 \leq i \leq N}$ are $\R$-valued measurable functions,
defined on a finite measure space $(\Omega_N, \cF_N, \p_N)$, that are uniformly bounded,
have a sign, say
\begin{equation} 
\label{eq:s0}
\exists\, s_0 > 0\colon\, \quad \p_N \big( \big\{ 0 \leq \sigma_i \leq s_0, \, 
\forall\, 1 \leq i \leq N \big\}^c \big) = 0 \qquad \forall\, N \in \N,
\end{equation}
and satisfy $-\infty < \limsup_{N\to\infty} \frac{1}{N} \log \p_N(\Omega_N) < \infty$.
\end{assumption}

\noindent
We emphasize that the $\sigma_i$'s need not be independent, nor exchangeable.
A more detailed discussion on Assumption~\ref{ass:model2} is given below.

We can now state the approximate version of \eqref{eq:tiltshift}. The free energy 
$\f(\beta,h;\delta)$ is again defined by \eqref{eq:freeen}.

\begin{theorem}[Asymptotic equivalence of tilting and shifting]
\label{th:compdeltah}
Subject to Assumptions~{\rm \ref{ass:disorder}} and {\rm \ref{ass:model2}}, and
with $\epsilon_0 := \min\{\frac{t_0}{2},\frac{t_0}{2s_0}\}$ (where $s_0, t_0$ are 
defined in \eqref{eq:s0} and \eqref{eq:t0}), for all $\beta \in [0, \epsilon_0)$ and 
$\delta \in (-\epsilon_0, \epsilon_0)$ there exist $0 < C^-_{\beta,\delta} \leq 
C^+_{\beta,\delta} < \infty$ such that
\begin{equation} 
\label{eq:compdeltah}
\begin{split}
& \forall\, \delta \in [0, \epsilon_0)\colon\,
\quad \f \big( \beta, h + C^{-}_{\beta,\delta} \, \beta \delta ; 0 \big) \leq
\f ( \beta, h; \delta ) \leq \f \big( \beta, h + C^{+}_{\beta,\delta} \, \beta \delta ; 0 \big),
\end{split}
\end{equation}
while for $\delta \in (-\epsilon_0, 0]$ the same relation holds with $C^{-}_{\beta,\delta}$ 
and $C^{+}_{\beta,\delta}$ interchanged. Moreover, $(\beta,\delta) \mapsto 
C^\pm_{\beta,\delta}$ is continuous with $C^\pm_{0,0} = 1$, and hence
\begin{equation}
\label{eq:propCpm}
\lim_{(\beta,\delta) \to (0,0)} C^\pm_{\beta,\delta} = 1.
\end{equation}
Furthermore, $\delta \mapsto C^\pm_{\beta,\delta}\,\delta$ is strictly increasing.
\end{theorem}

\noindent
The proof of Theorem~\ref{th:compdeltah} is given in Section~\ref{sec:tiltshift}. The 
general strategy and consists in showing that the derivatives of $\f(\beta,h;\delta)$ 
with respect to $\delta$ and $h$ are comparable. Compared to Giacomin and
Toninelli~\cite{cf:GT1}, several estimates need to be sharpened considerably.


\subsubsection{$\bullet $ Smoothing}

Combining Theorems~\ref{th:smoothing} and~\ref{th:compdeltah}, we finally obtain 
our smoothing inequality with respect to a shift, with explicit control on the constant.

\begin{theorem}[Smoothing inequality with respect to a disorder shift]
\label{th:smoshift}
Subject to Assumptions~{\rm \ref{ass:disorder}}, {\rm \ref{ass:model}} and 
{\rm \ref{ass:model2}}, there is an $\epsilon'_0 > 0$ with the following property:
if $\f(\bar\beta, \bar h; 0) = 0$ for some $\bar\beta \in (0,\epsilon'_0)$ and 
$\bar h \in \R$, then for $t \in (-\bar\beta \epsilon'_0, \bar\beta \epsilon'_0)$,
\begin{equation} 
\label{eq:compdeltah2}
0 \leq \f ( \bar\beta, \bar h + t; 0) \leq
\frac{\gamma}{2 \bar\beta^2} \, A_{\bar\beta,\frac{t}{\bar\beta}} \, t^2,
\end{equation}
where $(\beta,\delta) \mapsto A_{\beta,\delta}$ is continuous from $(0,\epsilon'_0) 
\times (\epsilon'_0, \epsilon'_0)$ to $(0,\infty)$, and is such that
\begin{equation}
	\lim_{(\beta,\delta) 
	\to (0,0)} A_{\beta,\delta} = 1 .
\end{equation}
\end{theorem}


\subsection{Discussion}
\label{sec:discussion}

We comment on the results obtained in Section~\ref{sec:theorems}.

\medskip\noindent
\textbf{1.}
The version of the smoothing inequality in Theorem~\ref{th:smoshift}, \emph{with the 
precision on the constant}, is picked up and used in Berger, Caravenna, Poisat, Sun 
and Zygouras~\cite{cf:BCPSZ} to obtain the sharp asymptotics of the critical 
curve $\beta \mapsto h_c(\beta)$ for pinning and copolymer models in the 
weak disorder regime $\beta \downarrow 0$, for the case $\alpha \in (1,\infty)$ (recall 
\eqref{eq:astau}).

\medskip\noindent
\textbf{2.} 
The smoothing inequality in \eqref{eq:compdeltah2},
\emph{at the level of generality at which it is stated}, is optimal in the following sense.
\begin{itemize}
\item 
We cannot hope for an exponent strictly larger than $2$ in the right-hand side of 
\eqref{eq:compdeltah2}, because pinning models with $\P(\tau_1 = n) \sim (\log n)/n^{3/2}$ 
are in the ``irrelevant disorder regime'', and it is known that $\f(\beta,h_c(\beta)+t; 0) 
\sim \f(0,h_c(0)+t; 0) = t^{2 + o(1)}$ as $t\downarrow 0$ for fixed $\beta > 0$ small 
enough (see Alexander \cite[Theorem 1.2]{cf:Ken}).
\item 
We cannot hope for an asymptotically smaller constant, i.e., $\lim_{(\beta,\delta) \to 
(0,0)} A_{\beta,\delta} < 1$, because the proof in Berger, Caravenna, Poisat, Sun and
Zygouras~\cite{cf:BCPSZ} would yield a contradiction (the lower bound would be 
strictly larger than the upper bound).
\end{itemize}
Of course, for specific models the inequality \eqref{eq:compdeltah2} can
sometimes be strengthened.
For instance, pinning models satisfying \eqref{eq:astau} with $\alpha \in (0,\frac{1}{2})$
are such that $\f(\beta,h_c(\beta)+t; 0) \sim \f(0,h_c(0)+t; 0) = t^{1/\alpha + o(1)}$
as $t \downarrow 0$ (see \eqref{eq:hompin}), again by Alexander \cite[Theorem 1.2]{cf:Ken}.

\medskip\noindent
\textbf{3.}
Compared with Assumption~\ref{ass:model}, Assumption~\ref{ass:model2} prescribes 
a specific form for the partition function $Z_{N, \omega, \beta, h}$ and therefore is 
more restrictive. On the other hand, in view of the minor constraints put on the $\sigma_i$'s, 
\eqref{eq:ZNgen} is so general that the absence of any restrictive conditions like 
\eqref{it:superadd} or \eqref{it:polylb} makes Assumption~\ref{ass:model2} effectively 
much weaker than Assumption~\ref{ass:model}. For instance, since \eqref{eq:model} 
is a special case of \eqref{eq:ZNgen}, with $\p_N (\cdot) = \p(\,\cdot\, \cap \{S_N = 0\})$
(which, incidentally, explains why $\p_N$ is allowed to be a finite measure, and not 
necessarily a probability), the model in \eqref{eq:model} satisfies Assumption~\ref{ass:model2}
as soon as the function $\phi$ is bounded and has a sign, without the need for any requirement 
like \eqref{eq:astau}. 

We emphasize that many other (also non-directed) disordered models 
fall into Assumption~\ref{ass:model2}. For instance, for $L \in\N$ set $\Lambda_L := \{-L, 
\ldots, +L\}^d$, $N := |\Lambda_L| = (2L+1)^d$, $\Omega_N := \{-1,+1\}^{\Lambda_L}$, 
and let $(\eta_i)_{i\in\Lambda_L}$ be the coordinate projections on $\Omega_N$. If $\p_N$ 
is the standard Ising Gibbs measure on $\Omega_N$, defined by $\p_N(\{\eta_i\}_{i\in\Lambda_L}) 
:= (1/Z_N) \exp[ J \sum_{i,j \in \Lambda_L, \,|i-j| = 1} \eta_i \eta_j]$, then the random variables 
$\sigma_i := \frac{1}{2}(\eta_i + 1)$ satisfy Assumption~\ref{ass:model2}.

\medskip\noindent
\textbf{4.} 
It follows easily from \eqref{eq:freeen} and \eqref{eq:ZNgen} that (with obvious notation)
\begin{equation} 
\label{eq:compa}
\f_{(\sigma_n + c)_{n\in\N}} ( \beta, h; \delta ) 
= \f_{(\sigma_n)_{n\in\N}} ( \beta, h; \delta ) + (\beta m_\delta + h) c.
\end{equation}
Therefore, when the $\sigma_n$'s are uniformly bounded but not necessarily non-negative, we can 
first perform a uniform translation to transform them into non-negative random variables, 
next apply \eqref{eq:compdeltah}, and finally use \eqref{eq:compa} to come back to the 
original $\sigma_n$'s. 

Still, the non-negativity assumption on the $\sigma_n$'s in \eqref{eq:s0}
cannot be dropped from Theorem~\ref{th:compdeltah}.
In fact, if $\f(\beta,h;\delta)$ is differentiable in $h$ and $\delta$, then
\eqref{eq:compdeltah} implies that
\begin{equation} 
\label{eq:tocheck}
\forall h \in \R\colon \qquad
\frac{\partial\f}{\partial \delta}(\beta,h ; 0) 
= \big[ 1 + o(1) \big]\,\beta \, \frac{\partial\f}{\partial h}(\beta,h ; 0),
\qquad \beta \downarrow 0.
\end{equation}
This relation, which is a necessary condition for \eqref{eq:compdeltah} when the free energy is 
differentiable, may be violated when the $\sigma_n$'s take both signs.
For instance, let $(\sigma_n)_{n\in\N}$ under $\p_N := \p$ be i.i.d.\ with
$\p(\sigma_n = -1) = \p(\sigma_n = +1) = \frac{1}{2}$, and let the marginal distribution 
of the disorder be $\bbP(\omega_n = -a^{-1}) = a^2/(a^2+1)$, $\bbP(\omega_n = a) 
= 1/(a^2+1)$ with $a > 0$ (note that $\bbE(\omega_1) = 0$ and $\bbvar(\omega_1) = 1$,
so that  \eqref{eq:t0} is satisfied). The free energy is easily computed:
\begin{equation}
\f(\beta,h ; \delta) = \bbE_\delta[ \cosh(h + \beta \omega_1) ]
= \frac{e^{a\delta} \cosh(h + a\beta) + a^2 e^{-a^{-1}\delta} \cosh(h - a^{-1}\beta)}
{e^{a\delta} + a^2 e^{-a^{-1}\delta}}.
\end{equation}
In particular,
\begin{align}
\frac{\partial\f}{\partial h} (\beta,0 ; 0) 
& = \frac{\sinh(a\beta) + a^2 \sinh(-a^{-1}\beta)}{1 + a^2} 
= \frac{a^2 - 1}{6a} \beta^3  + o(\beta^3),\\
\frac{\partial\f}{\partial \delta}(\beta, 0 ; 0) 
& = \frac{a \cosh(a\beta) - a \cosh(- a^{-1}\beta)}{1 + a^2} 
= \frac{a^2 - 1}{2a} \, \beta^2 \,+\, o(\beta^2),
\end{align}
and hence \eqref{eq:tocheck} does \emph{not} hold for $a \ne 1$ (the left-hand side 
is $\approx \beta^2$, while the right-hand side is $\approx \beta^4$). 
Intuitively, such a discrepancy
arises for values of $h$ at which $\frac{\partial\f}{\partial h}(0,h;0) = 0$, which means that
the average $\e_{N,\omega, 0,h} (\frac{1}{N}\sum_{n=1}^N \sigma_n)$
tends to zero as $N\to\infty$,
where $\p_{N,\omega, \beta,h}$ is the Gibbs law
associated to the partition function $Z_{N, \omega, \beta, h}$
(see \eqref{eq:PZ} below).
When the $\sigma_n$'s are non-negative, their individual 
variances under $\p_{N,\omega, 0,h}$ must be small, but this is no longer true when 
the $\sigma_n$'s can also take negative values.
This is why one might have
$\frac{\partial\f}{\partial \delta}(\beta,h;0) \gg \beta\frac{\partial\f}{\partial h}(\beta,h;0)$
for $\beta > 0$ small
(compare \eqref{eq:parth} with \eqref{eq:fnoy}-\eqref{eq:pias} below).


\smallskip

\section{Smoothing with respect to a tilt: proof of Theorem~\ref{th:smoothing}}
\label{sec:protilt}


\subsection{The $(\cG,\cC)$-rare stretch strategy}
\label{sec:rare-stretch-strategy}

Fix $\beta \geq 0$ and $h \in \R$. For $\ell\in\N$, let 
$\cA_{\ell} \subseteq \R^{\ell}$ be a subset of ``disorder stretches'' such that 
there exist constants $\cG \in [0,\infty)$ and $\cC \in [0,\infty)$ with the following 
properties, along a diverging sequence of $\ell \in \N$:
\begin{itemize}
\item 
$\frac{1}{\ell} \log Z_{\ell,\omega,\beta,h} \geq \cG$ for all $\omega = \omega_{(0,\ell]} 
:= (\omega_1, \ldots, \omega_{\ell}) \in \cA_\ell$ (recall Assumption~\ref{ass:model}
\eqref{it:basic});
\item $\frac{1}{\ell} \log \bbP(\cA_\ell) \geq - \cC$.
\end{itemize}
The notation $(\cG,\cC)$ stands for \emph{gain} versus \emph{cost}. Recall that 
$\gamma$ is the exponent in \eqref{eq:polylb}.

\begin{lemma}
\label{lem:CG}
The following implication holds:
\begin{equation} 
\label{eq:loccond}
\cG - \gamma\, \cC \, >  0 \quad \Longrightarrow  \quad \f(\beta, h ; 0) > 0.
\end{equation}
\end{lemma}

\begin{proof}
Fix $\ell \in \N$ large enough so that the above conditions hold, and for $\omega \in \R^\N$ 
denote by $T_1(\omega), T_2(\omega), \ldots$ the distances between the endpoints of 
the stretches in $\cA_\ell$:
\begin{equation}
T_1(\omega) := \inf\big\{N \in \ell\N\colon\, \omega_{(N-\ell,
N]} \in \cA_{\ell}\big\},
\qquad 
T_{k+1}(\omega) := T_{1}(\theta^{T_{1}(\omega) + \ldots + T_{k}(\omega)}(\omega)).
\end{equation}
Note that $\{T_k\}_{k\in\N}$ is i.i.d.\ with marginal law given by $\ell \,\mathrm{GEO}
(\bbP(\cA_{\ell}))$. In particular,
\begin{equation}
	\bbE(T_1) = \ell / \bbP(\cA_{\ell}) \le \ell \, e^{\cC \ell}.
\end{equation}
Henceforth we suppress the subscripts $\beta,h$. 
Since $(\theta^{(T_1 + \ldots + T_{i})-\ell} \omega)_{(0,\ell]} \in \cA_\ell$
by construction, applying
properties \eqref{it:superadd}-\eqref{it:polylb} in Assumption~\ref{ass:model}
and the definition of $\cG$, we get
\begin{equation}
Z_{T_1 + \ldots + T_k,\omega} 
\geq \prod_{i=1}^k Z_{T_i - \ell, \theta^{(T_1 + \ldots + T_{i-1})} \omega}
\, Z_{\ell, \theta^{(T_1 + \ldots + T_{i})-\ell} \omega}
\geq e^{k \cG \ell} \, \prod_{i=1}^k 
\frac{c_{\beta,h}(\theta^{(T_1 + \ldots + T_{i-1})}\omega)}{(T_i)^\gamma},
\end{equation}
where we set $Z_0 := 1$ for convenience.
Recalling \eqref{eq:freeen} and Remark~\ref{rem:kingman},
for $\bbP$-a.e. $\omega$ we can write, by the strong 
law of large numbers and Jensen's inequality,
\begin{equation}
\begin{split}
\f(\beta, h ; 0) 
& 
= \lim_{k\to\infty} \frac{1}{T_1 + \ldots + T_k}
\log Z_{T_1 + \ldots + T_k,\omega,\beta,h} \\
& \geq \frac{1}{\bbE(T_1(\omega))} \big\{ \ell \cG + \bbE[\log c_{\beta,h}(\omega)]
- \gamma \, \bbE[\log (T_1) ] \big\} \\
& \geq \frac{1}{\bbE(T_1(\omega))} \big\{ \ell \cG + \bbE[\log c_{\beta,h}(\omega)]
- \gamma \, \log \bbE(T_1) \big\} \\
& = e^{-\cC \ell} \bigg\{ (\cG - \gamma \cC )
+ \frac{\bbE[\log c_{\beta,h}(\omega)]}{\ell} - \gamma \frac{\log \ell}{\ell} \bigg\}.
\end{split}
\end{equation}
If $\cG - \gamma \cC > 0$, then we can choose $\ell \in \N$ large enough (but finite!) 
such that the right-hand side is strictly positive. This proves \eqref{eq:loccond}.
\end{proof}


\subsection{Proof of Theorem~\ref{th:smoothing}}
We use Lemma~\ref{lem:CG}. Fix $\beta > 0$, $h \in \R$,
$\delta \in (-t_0, t_0)$ and $\epsilon > 0$, 
and define the set of good atypical stretches as
\begin{equation}
\cA_{\ell} := \bigg\{ (\omega_1, \ldots, \omega_\ell) \in \R^\ell\colon\, 
\frac{1}{\ell}\log Z_{\ell, \omega, \beta, h} \geq \f(\beta, h; \delta) - \epsilon \bigg\},
\end{equation}
so that $\cG = \f(\beta, h; \delta) - \epsilon $ by construction. It remains to determine $\cC$, for 
which we need to estimate the probability of $\bbP(\cA_\ell)$ from below. 

By the definition \eqref{eq:freeen} of $\f(\beta, h; \delta)$ together with Kingman's super-additive
ergodic theorem (see Remark~\ref{rem:kingman}),
the event $\cA_\ell$ is typical for ${ \bbP}_\delta$:
\begin{equation} 
\label{eq:pdelta1}
\lim_{\ell\to+\infty} { \bbP}_\delta(\cA_\ell) = 1.
\end{equation}
Denoting by $\bbP_\delta^\ell$ (resp. $\bbP^\ell$) the restriction of $\bbP_\delta$
(resp. $\bbP$) on $\sigma(\omega_1, \ldots, \omega_\ell)$, we have, by Jensen's inequality
and \eqref{eq:RNn},
\begin{equation} \label{eq:entrin}
\begin{split}
	\bbP(\cA_\ell) & = \bbP_\delta(\cA_\ell) \,
	\bbE_\delta\bigg( e^{- \log \frac{\dd \bbP^\ell_\delta}{\dd\bbP^\ell}}
	\bigg| \cA_\ell \bigg) \ge
	\bbP_\delta(\cA_\ell) \,
	e^{- \bbE_\delta \big(\log \frac{\dd \bbP^\ell_\delta}{\dd\bbP^\ell}
	\big| \cA_\ell \big)} \\
	& = \bbP_\delta(\cA_\ell) \,
	e^{- \frac{1}{\bbP_\delta(\cA_\ell)}
	\bbE_\delta \big[ \big(
	\log \frac{\dd \bbP^\ell_\delta}{\dd\bbP^\ell}\big)  \, \ind_{\cA_\ell} \big]} \\
	& =  \bbP_\delta(\cA_\ell) \,
	e^{- \frac{\ell}{\bbP_\delta(\cA_\ell)}
	\bbE_\delta \big[ \big(
	\delta \frac{\omega_1 + \ldots + \omega_\ell}{\ell} - \log \M(\delta)
	\big)  \, \ind_{\cA_\ell} \big]}\,.
\end{split}
\end{equation}
Recalling \eqref{eq:RNn} and Assumption~\ref{ass:disorder}, we abbreviate
\begin{equation}
\label{eq:mdelta0}
m_\delta := \bbE_\delta(\omega_1) = (\log\M)'(\delta)
= \delta + o(\delta), \qquad \delta \to 0.
\end{equation}
By the strong law of large numbers, it follows from \eqref{eq:pdelta1}-\eqref{eq:entrin}
that for every $\epsilon > 0$ we have, for $\ell$ large enough,
\begin{equation}
\frac{1}{\ell} \log \bbP(\cA_\ell) \geq 
- \big[ \delta \, m_\delta - \log \M(\delta) \big] \,-\, \epsilon =: -\cC,
\end{equation}


We can conclude.
We know from \eqref{eq:loccond} that $\f(\beta, h; 0) > 0$ when
\begin{equation}
\cG - \gamma \cC = \f(\beta, h; \delta) - \gamma \big[ \delta \, m_\delta 
- \log \M(\delta) \big] - 2 \epsilon > 0.
\end{equation}
If $\f(\bar\beta, \bar h) = 0$, as in the assumptions of Theorem~\ref{th:smoothing},
it follows that $\cG - \gamma \cC \leq 0$, 
i.e.,
\begin{equation}
\f(\bar\beta, \bar h; \delta) \leq \gamma 
\big[ \delta \, m_\delta - \log \M(\delta) \big] + 2 \epsilon, \qquad
\forall \delta \in (-t_0, t_0) \,.
\end{equation}
Since this equality holds for every $\epsilon > 0$, it must hold also for $\epsilon = 0$,
proving \eqref{eq:smodelta}.
\qed


\smallskip

\section{Asymptotic equivalence of tilting and shifting: proof of
Theorem~\ref{th:compdeltah}}
\label{sec:tiltshift}

Throughout this section, we work under Assumptions~{\rm \ref{ass:disorder}} 
and {\rm \ref{ass:model2}}.


\subsection{Notation}
Denote the empirical average of the variables $\sigma_i$'s by
\begin{equation} 
\label{eq:sigmabar}
\overline \sigma_N := \frac{1}{N} \sum_{i=1}^N \sigma_i.
\end{equation}
The finite-volume Gibbs measure associated with the partition function in 
\eqref{eq:ZNgen} is the \emph{probability} on $\Omega_N$
defined, for $N\in\N$, $\omega \in \R^\N$, $\beta \geq 0$ 
and $h\in\R$, by
\begin{equation}\label{eq:PZ}
\begin{split}
& \p_{N, \omega, \beta, h}(\,\cdot\,) 
:= \frac{1}{Z_{N, \omega, \beta, h}} 
\e_N \Big[e^{\sum_{n=1}^N (h + \beta \omega_n) \sigma_n} \, \ind_{\{\cdot\}} \Big] \,, \\
& \text{where} \quad
Z_{N, \omega, \beta, h} := \e_N \Big[ e^{\sum_{n=1}^N (h + \beta \omega_n) \sigma_n} \Big]
\,. \end{split}
\end{equation}
Let us spell out the definition \eqref{eq:freeen} of the free energy, recalling \eqref{eq:RNn}:
\begin{equation} 
\label{eq:freee}
\begin{split}
\f(\beta,h; \delta) 
&:= \limsup_{N\to\infty} \f_N(\beta,h; \delta) := \limsup_{N\to\infty} \frac{1}{N} 
\bbE_\delta\big[ \log Z_{N, \omega, \beta, h} \big]  \\
&= \limsup_{N\to\infty} \frac{1}{N} \bbE \big[ e^{\sum_{n=1}^N [\delta \omega_n
- \log\M(\delta)]} \log Z_{N, \omega, \beta, h} \big].
\end{split}
\end{equation}
Note that, by \eqref{eq:s0},
\begin{equation} 
\label{eq:boundham}
\Bigg| \sum_{n=1}^N (h + \beta \omega_n) \sigma_n \Bigg| 
\leq \sum_{n=1}^N (|h| + \beta |\omega_n|) \, |\sigma_n|
\leq s_0 \sum_{n=1}^N (|h| + \beta |\omega_n|),
\end{equation}
so that $|\f(\beta,h;\delta)| \leq s_0 (|h| + \beta \bbE_\delta(|\omega_1|))
\,+\, |\limsup_{N\to\infty} \frac{1}{N} \log \p_N(\Omega_N)| < \infty$.


\subsection{Preparation}

Before proving Theorem~\ref{th:compdeltah}, we need some preparation. Recalling 
\eqref{eq:sigmabar}, we define for $[a,b] \subseteq \R$ with $a<b$ a restricted version 
of the partition function and the free energy, in which the empirical average
$\overline\sigma_N$ is constrained to lie in $[a,b]$:
\begin{equation} 
\label{eq:freeeab}
\begin{split}
Z_{N, \omega, \beta, h}^{[a,b]} 
& := \e_N \Big[ e^{\sum_{n=1}^N (h + \beta \omega_n) \sigma_n} \, 
\ind_{\{\overline\sigma_N \in [a,b]\}} \Big],\\
\f^{[a,b]}(\beta,h; \delta) 
& := \limsup_{N\to\infty} \f^{[a,b]}_N(\beta,h; \delta) 
:= \limsup_{N\to\infty} \frac{1}{N} \bbE_\delta\big[ \log Z_{N, \omega, \beta, h}^{[a,b]} \big].
\end{split}
\end{equation}
The corresponding restricted Gibbs measure is the probability defined by (recall \eqref{eq:PZ})
\begin{equation} 
\label{eq:pNab}
\p_{N, \omega, \beta, h}^{[a,b]} (\,\cdot\,) 
:= \p_{N, \omega, \beta, h} (\,\cdot\,|\,\overline\sigma_N \in [a,b] )
= \frac{\e_N \Big[ e^{\sum_{n=1}^N (h + \beta \omega_n) \sigma_n} \, 
\ind_{\{\overline\sigma_N \in [a,b]\}} \, \ind_{\{\,\cdot\,\}} \Big]}{Z_{N, \omega, \beta, h}^{[a,b]}}.
\end{equation}

Note that $Z_{N, \omega, \beta, h} = Z_{N, \omega, \beta, h}^{[0,s_0]}$, by \eqref{eq:s0}.
Furthermore, $Z_{N, \omega, \beta, h}^{[a,b]} \le Z_{N, \omega, \beta, h}^{[c,d]}$ when 
$[a,b] \subseteq [c,d]$. Therefore
\begin{equation} 
\label{eq:monf}
\f^{[a,b]}(\beta,h;\delta) \leq \f^{[c,d]}(\beta,h;\delta) \,\le\, \f (\beta,h;\delta),
\qquad [a,b] \subseteq [c,d].
\end{equation}
In particular, for $x \in \R$ we may define
\begin{equation} 
\label{eq:fx}
\f^{\{x\}}(\beta,h;\delta) := \lim_{n \to \infty} \f^{[a_n, b_n]}(\beta,h;\delta)
\in [-\infty, +\infty),
\end{equation}
where $a_n \uparrow x$ and $b_n \downarrow x$ are arbitrary strictly monotone 
sequences (it is easily seen that the limit does not depend on the choice of these 
sequences). 

Note that $\f^{\{x\}}(\beta,h;\delta) = -\infty$ when $x \not\in [0,s_0]$, 
by \eqref{eq:s0}. The following result is standard:
\begin{equation}
\label{eq:fbetah}
\f(\beta,h;\delta) = \sup_{x \in [0,s_0]}  \f^{\{x\}}(\beta,h;\delta).
\end{equation}
In fact,
by \eqref{eq:monf} $\f^{[a,b]}(\beta,h;\delta) \le \f(\beta,h;\delta)$ for every $[a,b] \subseteq \R$,
hence by \eqref{eq:fx} $\f(\beta,h;\delta) \geq \f^{\{x\}}(\beta,h;\delta)$ for every 
$x\in\R$. It follows that the inequality $\geq$ holds in \eqref{eq:fbetah}. For the reverse inequality, 
note that if $a < b < c$, then $[a,c] \subseteq [a,b] \cup [b,c]$ and so
\begin{equation}
Z_{N, \omega, \beta, h}^{[a,c]} 
\leq Z_{N, \omega, \beta, h}^{[a,b]} + Z_{N, \omega, \beta, h}^{[b,c]} 
\leq 2 \, \max \Big\{ Z_{N, \omega, \beta, h}^{[a,b]},Z_{N, \omega, \beta, h}^{[b,c]} \Big\}.
\end{equation}
Recalling \eqref{eq:freeeab}, we see that
\begin{equation}
\f^{[a,c]}(\beta,h;\delta) \leq 
\max \Big\{ \f^{[a,b]}(\beta,h;\delta), \, \f^{[b,c]}(\beta,h;\delta) \Big\}.
\end{equation}
Since $\f(\beta,h;\delta) = \f^{[0,s_0]}(\beta,h;\delta)$, we can build a sequence of closed 
intervals $(I_n)_{n\in\N_0}$, where $I_0 = [0,s_0]$ and where $I_{n+1}$ is either the first 
half or the second half of $I_n$, such that $\f^{I_n}(\beta,h;\delta) \leq \f^{I_{n+1}}(\beta,h;\delta)$ 
for all $n\in\N$. In particular,
\begin{equation}
\f(\beta,h;\delta) \leq \lim_{n\to\infty} \f^{I_n}(\beta,h;\delta).
\end{equation}
By compactness, there exists an $\overline x \in [0,s_0]$ such that $I_n \downarrow 
\{\overline x\}$, i.e., $\bigcap_{n\in\N} I_n = \{\overline x\}$. If $I_n = [a_n, b_n]$, then we
set $J_n := [a_n - \frac{1}{n}, b_n + \frac{1}{n}]$, so that we still have $J_n \downarrow 
\{\overline x\}$, and $\overline x$ lies in the interior of each $J_n$. Since $\f^{I_n}(\beta,h;\delta) 
\leq \f^{J_n}(\beta,h;\delta)$, recalling \eqref{eq:fx} we obtain
\begin{equation}
\f(\beta,h;\delta) \leq \lim_{n\to\infty} \f^{I_n}(\beta,h;\delta)
\leq \lim_{n\to\infty} \f^{J_n}(\beta,h;\delta) 
= \f^{\{\overline x\}}(\beta,h;\delta) 
\leq \sup_{x \in [0,s_0]}  \f^{\{x\}}(\beta,h;\delta),
\end{equation}
and the proof of \eqref{eq:fbetah} is complete.


\subsection{Proof of Theorem~\ref{th:compdeltah}}

By \eqref{eq:fbetah}, it suffices to show that \eqref{eq:compdeltah} is satisfied
with $\f^{\{x\}}$ instead of $\f$, for every fixed $x \in [0, s_0]$. It is of course 
important that the constants $C^\pm_{\beta,\delta}$ do not depend on $x$.

\medskip\noindent
\textbf{1.}
First we consider the case $x = 0$. We claim that
\begin{equation} 
\label{eq:f0}
\f^{\{0\}}(\beta,h;\delta) = \lim_{\epsilon \downarrow 0}
\bigg( \limsup_{N\to\infty} \frac{1}{N} \log \p_N(0 \le \overline\sigma_N \le \epsilon) \bigg).
\end{equation}
Since the right-hand side of \eqref{eq:f0} is a constant that does not depend on $\beta \ge 0$,
$\delta \in (-t_0,t_0)$ and $h\in\R$, \eqref{eq:compdeltah} is trivially satisfied
with $\f^{\{0\}}$ instead of $\f$, whatever the 
definition of $C^\pm_{\beta,\delta}$ is. To prove \eqref{eq:f0} note that, by Cauchy-Schwarz,
\begin{equation}
\begin{split}
\Bigg| \sum_{n=1}^N (h + \beta \omega_n) \sigma_n \Bigg| 
& \leq \sqrt{\sum_{n=1}^N (h + \beta \omega_n)^2} \, \sqrt{\sum_{n=1}^N |\sigma_n|^2} 
\leq N \, s_0 \, \sqrt{\overline\sigma_N} \, \sqrt{\frac{1}{N} \sum_{n=1}^N (h + \beta \omega_n)^2},
\end{split}
\end{equation}
because $0 \leq \sigma_n = |\sigma_n| \leq s_0$ by \eqref{eq:s0}. Recalling \eqref{eq:freeeab}, 
for every $N\in\N$ we get
\begin{equation}
\begin{split}
\bigg| \frac{1}{N} \bbE_\delta \big[ \log Z^{[a,b]}_{N,\omega,\beta,h} \big] 
- \frac{1}{N} \log \p_N(a \le \overline\sigma_N \le b) \bigg| 
&\leq  s_0 \, \sqrt{b} \, \bbE_\delta \left[ \sqrt{\frac{1}{N} 
\sum_{n=1}^N (h + \beta \omega_n)^2} \right] \\
& \leq s_0 \, \sqrt{b} \,\sqrt{\bbE_\delta \big[ (h + \beta \omega_1)^2 \big]},
\end{split}
\end{equation}
where we use Jensen. Note that the right-hand side is a finite constant. If $|a_N - b_N| \leq c$ for 
all $N\in\N$, then $|\limsup_N a_N - \limsup_N b_N| \leq c$, and so
\begin{equation}
\bigg| \f^{[a,b]}(\beta,h;\delta) 
-  \bigg(\limsup_{N\to\infty} \frac{1}{N} \log \p_N(a \le \overline\sigma_N \le b) \bigg) \bigg|
\leq s_0 \, \sqrt{b} \,\sqrt{\bbE_\delta \big[ (h + \beta \omega_1)^2 \big]}.
\end{equation}
Taking $[a,b] = [-\epsilon, \epsilon]$ and letting $\epsilon \downarrow 0$, we get \eqref{eq:f0}
from \eqref{eq:fx} .

\medskip\noindent
\textbf{2.}
Next we consider the case $x \in (0, s_0]$. Roughly speaking, the strategy of the proof is to 
show that the derivatives of the free energy with respect to $\delta$ and to $h$ are comparable.
Unless otherwise specified, we work with generic values of the parameters in the admissible 
range $\beta \geq 0$, $h\in\R$ and $\delta \in (-t_0, t_0)$. Henceforth we fix $0 < a < b < \infty$. 
Recalling \eqref{eq:freeeab} and \eqref{eq:pNab}, we see that the derivative with respect to $h$ 
of the (restricted) finite-volume free energy $\f^{[a,b]}_N(\beta, h;\delta)$ can be expressed as
\begin{equation} 
\label{eq:parth}
\frac{\partial}{\partial h} \f^{[a,b]}_N(\beta, h;\delta) 
= \frac{1}{N} \bbE_\delta\bigg[ \frac{\partial}{\partial h} \log Z^{[a,b]}_{N, \omega, \beta, h} \bigg] 
= \bbE_\delta \big[ \e^{[a,b]}_{N, \omega, \beta, h}\big[ \overline\sigma_N \big] \big].
\end{equation}

\medskip\noindent
\textbf{3.}
The derivative with respect to $\delta$ requires some further estimates. Recalling
\eqref{eq:PZ}-\eqref{eq:freee}, we have
\begin{equation} 
\label{eq:start}
\frac{\partial}{\partial\delta} \f^{[a,b]}_N(\beta, h; \delta) 
= \frac{1}{N} \sum_{n=1}^N \bbE_\delta \Big[ (\omega_n - m_\delta) \,
\log Z^{[a,b]}_{N, \omega, \beta, h} \Big],
\end{equation}
where $m_\delta := \bbE_\delta(\omega_n) = (\log \M)'(\delta)$ by \eqref{eq:mdelta0}.
Subtracting a centering term with zero mean, we get
\begin{equation} 
\label{eq:ddf}
\begin{split}
\frac{\partial}{\partial\delta} \f^{[a,b]}_N(\beta, h; \delta) 
& = \frac{1}{N} \sum_{n=1}^N \bbE_\delta \Big[ (\omega_n - m_\delta) \,
\Big( \log Z^{[a,b]}_{N, \omega, \beta, h} 
- \log Z^{[a,b]}_{N, \omega, \beta, h}|_{\omega_n = m_\delta} \Big) \Big] \\
& = \frac{1}{N} \sum_{n=1}^N \bbE_\delta \Bigg[ (\omega_n - m_\delta) 
\int_{m_\delta}^{\omega_n} \bigg( \frac{\partial}{\partial\omega_n} 
\log Z^{[a,b]}_{N, \omega, \beta, h} \bigg) \bigg|_{\omega_n = y} \, \dd y \Bigg],
\end{split}
\end{equation}
where we agree that $\int_a^b (\ldots) := -\int_{b}^a(\ldots)$ when $a > b$.
Abbreviate
\begin{equation} 
\label{eq:fnoy}
f_n(\omega, y) := \frac{1}{\beta} \, \bigg( \frac{\partial}{\partial \omega_n} 
\log Z^{[a,b]}_{N, \omega, \beta, h} \bigg)\bigg|_{\omega_n = y}
= \e^{[a,b]}_{N, \omega, \beta, h}|_{\omega_n = y}[\sigma_n],
\end{equation}
where the second equality follows easily from \eqref{eq:freeeab} via \eqref{eq:pNab}.
Note that $f_n(\omega,y)$ depends on the $\omega_i$'s for $i \ne n$, not on 
$\omega_n$. Therefore \eqref{eq:ddf} can be rewritten as
\begin{equation} 
\label{eq:pias}
\frac{\partial}{\partial\delta} \f^{[a,b]}_N(\beta, h; \delta) 
= \frac{\beta}{N} \sum_{n=1}^N \bbE_\delta \Big[ (\omega_n - m_\delta)^2 \,
\frac{1}{\omega_n - m_\delta} \int_{m_\delta}^{\omega_n} f_n(\omega,y) \, \dd y \Big].
\end{equation}

\medskip\noindent
\textbf{4.}
By \eqref{eq:fnoy}, the integral average in \eqref{eq:pias} should be close to
$\e^{[a,b]}_{N, \omega, \beta, h}[\sigma_n]$. If we could factorize the expectation
over $\bbE_\delta$, then the right-hand side in \eqref{eq:pias} would become 
$\approx \beta \, \bbvar_\delta(\omega_1)\, \bbE_\delta \big[ \e^{[a,b]}_{N, \omega, 
\beta, h}[\overline\sigma_N] \big]$. Recalling \eqref{eq:parth}, we see that this is 
precisely what we want, because $\bbvar_\delta(\omega_1) \approx 1$ for $\delta$ 
small. In order to turn these arguments into a proof, we need to estimate the dependence 
of $f_n(\omega, y)$ on $y$. To that end we note that
\begin{equation}
\label{eq:varus2}
\begin{split}
\frac{\partial}{\partial \omega_n} f_n(\omega, \omega_n) 
& = \frac{1}{\beta} \frac{\partial^2}{\partial \omega_n^2}
\log Z^{[a,b]}_{N, \omega, \beta, h} \,=\, \beta \, 
\var_{N, \omega, \beta, h}^{[a,b]}[\sigma_n] \\
& \leq \beta \, \e^{[a,b]}_{N, \omega, \beta, h} [\sigma_n^2] 
\leq  s_0 \, \beta \, \e^{[a,b]}_{N, \omega, \beta, h}[\sigma_n] 
= s_0 \, \beta \,  f_n(\omega, \omega_n)
\end{split}
\end{equation}
because $0 \le \sigma_n \le s_0$, by \eqref{eq:s0}. Therefore
\begin{equation}
\frac{\partial}{\partial y} f_n(\omega, y) \geq 0, 
\qquad \frac{\partial}{\partial y} \big( e^{-s_0 \, \beta \, y} \, f_n(\omega, y) \big) \leq	0,
\end{equation}
and integrating these relations we get
\begin{equation} 
\label{eq:ustec}
e^{-s_0 \beta  (y - y')^-} \, f_n(\omega, y') 
\leq f_n(\omega, y) \leq e^{s_0 \beta  (y - y')^+} \, f_n(\omega, y') 
\qquad \forall\, y,y' \in \R.
\end{equation}
Introducing the function
\begin{equation} 
\label{eq:defg}
g(x) := \begin{cases}
\displaystyle
\frac{e^{x} - 1}{x} 
& \text{if } x \ne 0, \\
1 
& \text{if } x = 0,
\end{cases}
\end{equation}
taking $y' = m_\delta$ in \eqref{eq:ustec} and integrating over $y$, we easily obtain the bounds
\begin{equation} 
\label{eq:g-+est}
\begin{split}
g\big(-\beta s_0(\omega_n - 
& m_\delta)^-\big) \, f_n(\omega, m_\delta) \\
& \leq \frac{1}{\omega_n - m_\delta}
\int_{m_\delta}^{\omega_n} f_n(\omega,y) \, \dd y 
\leq g\big(\beta s_0(\omega_n - m_\delta)^+\big) \, f_n(\omega, m_\delta).
\end{split}
\end{equation}

\medskip\noindent
\textbf{5.}
Before inserting this estimate into \eqref{eq:pias}, let us pause for a brief integrability interlude.
The random variable $g(-\beta s_0(\omega_n - m_\delta)^-)$ is bounded, so there is no integrability 
concern. On the other hand, the random variable $g(\beta s_0(\omega_n - m_\delta)^+)$ is 
unbounded and a little care is required. Note that
\begin{equation}
g(\beta s_0(\omega_n - m_\delta)^+) \leq A + B \, e^{\beta s_0 \omega_n }
\end{equation}
for $A,B>0$, and that $\bbE_\delta(e^{t\omega_1}) < \infty$ for $t + \delta \in (-t_0, +t_0)$, by 
\eqref{eq:t0} and \eqref{eq:RNn}. Therefore, when we integrate  $g(\beta s_0(\omega_n - 
m_\delta)^+)$ (possibly times a polynomial of $\omega_n$) over $\bbP_\delta$, to have a 
finite outcome we need to ensure that $\beta s_0 + \delta \in (-t_0, +t_0)$. This is simply 
achieved through the restrictions $\delta \in (-\epsilon_0 ,\epsilon_0)$ and $\beta \in [0,\epsilon_0)$,
where $\epsilon_0 := \min\{\frac{t_0}{2},\frac{t_0}{2s_0}\}$,
as in the statement of Theorem~\ref{th:compdeltah}.
We make these restrictions henceforth.

\medskip\noindent
\textbf{6.}
Let us now substitute the estimate \eqref{eq:g-+est} into \eqref{eq:pias}. Since $f_n(\omega, 
m_\delta)$ does not depend on $\omega_n$, the expectation over $\bbE_\delta$ factorizes 
and we obtain
\begin{equation} 
\label{eq:aatt}
\begin{split}
\bbE_\delta \Big[ (\omega_1 - m_\delta)^2 \, 
& g\big(-\beta s_0 (\omega_1 - m_\delta)^-\big) \Big]
\, \Bigg( \frac{\beta}{N} \sum_{n=1}^N
\bbE_\delta \Big[ f_n(\omega, m_\delta) \Big] \Bigg) \\
& \leq \frac{\partial}{\partial\delta} \f^{[a,b]}_N(\beta, h; \delta) \\
& \leq \bbE_\delta \Big[ (\omega_1 - m_\delta)^2 \, g\big(\beta s_0 (\omega_1 - m_\delta)^+\big) \Big]
\, \Bigg( \frac{\beta}{N} \sum_{n=1}^N \bbE_\delta \Big[ f_n(\omega, m_\delta) \Big] \Bigg).
\end{split}
\end{equation}
We next want to replace $f_n(\omega, m_\delta)$ by $f_n(\omega, \omega_n) 
= \e^{[a,b]}_{N, \omega, \beta, h} \big[ \sigma_n \big]$ (recall \eqref{eq:fnoy}). 
To this end, we again apply \eqref{eq:ustec}, this time with $y = \omega_n$ and 
$y' = m_\delta$. Since $f_n(\omega, m_\delta)$ does not depend on $\omega_n$, 
we have
\begin{equation} 
\label{eq:aatt+}
\begin{split}
\frac{\bbE_\delta \big[ f_n(\omega, \omega_n) \big]}
{\bbE_\delta \big[e^{s_0 \beta  (\omega_1 - m_\delta)^+ } \big]}
\leq \bbE_\delta \Big[ f_n(\omega, m_\delta) \Big] 
& \leq \frac{\bbE_\delta \big[ f_n(\omega, \omega_n) \big]}
{\bbE_\delta \big[ e^{-s_0 \beta (\omega_1 - m_\delta)^-} \big]}.
\end{split}
\end{equation}
We can now introduce the constants
\begin{equation} 
\label{eq:c+-}
\begin{split}
c^+_{\beta, \delta} 
& := \frac{\bbE_\delta \big[ (\omega_1 - m_\delta)^2 
\, g\big(\beta s_0 (\omega_1 - m_\delta)^+\big) \big]}
{\bbE_\delta \big[ e^{-s_0 \beta (\omega_1 - m_\delta)^-} \big]}, \\
c^-_{\beta, \delta} 
& := \frac{\bbE_\delta \big[ (\omega_1 - m_\delta)^2 
\, g\big(-\beta s_0 (\omega_1 - m_\delta)^-\big) \big]}
{\bbE_\delta \big[ e^{s_0 \beta  (\omega_1 - m_\delta)^+ } \big]},
\end{split}
\end{equation}
and note that $0 < c^-_{\beta, \delta} \le c^+_{\beta, \delta} < \infty$ for all $\delta 
\in (-\epsilon_0, \epsilon_0)$ and $\beta \in [0,\epsilon_0)$. We have already observed 
that $f_n(\omega, \omega_n) = \e^{[a,b]}_{N, \omega, \beta, h} \big[ \sigma_n \big]$
by \eqref{eq:fnoy}, and so from \eqref{eq:aatt}-\eqref{eq:aatt+} we obtain the following 
estimate: for every $\beta \in [0,\epsilon_0)$, $h\in\R$, $\delta \in (-\epsilon_0, \epsilon_0)$ 
and $0 < a < b <\infty$
\begin{equation} 
\label{eq:partdelta}
c^-_{\beta,\delta} \, \beta \, \bbE_\delta \big[ \e^{[a,b]}_{N, \omega, \beta, h}
\big[ \overline\sigma_N \big] \big] 
\leq \frac{\partial}{\partial\delta} \f^{[a,b]}_N(\beta, h; \delta)
\leq c^+_{\beta,\delta} \, \beta \, \bbE_\delta \big[ \e^{[a,b]}_{N, \omega, \beta, h}
\big[ \overline\sigma_N \big] \big].
\end{equation}
Note the analogy with the expression in \eqref{eq:parth} for $\frac{\partial}{\partial h}
\f^{[a,b]}_N(\beta, h; \delta)$.

\medskip\noindent
\textbf{7.}
We are close to the final conclusion. Since by \eqref{eq:pNab} we have $a \leq 
\e^{[a,b]}_{N, \omega, \beta, h}\big[ \overline\sigma_N \big] \leq b$, it follows from 
\eqref{eq:partdelta} that, for every $\delta \in [0,\epsilon_0)$
\begin{equation} 
\label{eq:cru0}
C^-_{\beta,\delta} \, \beta \, a \, \delta 
\leq \f^{[a,b]}(\beta, h; \delta) - \f^{[a,b]}(\beta, h; 0) 
\leq C^+_{\beta,\delta} \, \beta \, b \, \delta,
\end{equation}
where we set
\begin{equation}
\label{eq:C+-}
C^\pm_{\beta,\delta} := 
\begin{cases}
\displaystyle \frac{1}{\delta} \int_0^\delta 
c^\pm_{\beta,\delta'} \, \dd \delta' 
& \text{if } \delta \in (-\epsilon_0, \epsilon_0) \setminus\{0\}, \\
c^\pm_{\beta,0} 
& \text{if } \delta = 0.
\end{cases}
\end{equation}
Analogously to \eqref{eq:cru0}, from \eqref{eq:parth} we obtain, for every $\xi \geq 0$,
\begin{equation} 
\label{eq:cru00}
a \, \xi \leq \f^{[a,b]}(\beta, h + \xi ; 0) -
\f^{[a,b]}(\beta, h ; 0) \leq b \, \xi.
\end{equation}
Choosing $\xi = C^+_{\beta,\delta} \frac{b}{a} \beta \delta$ and $\xi = C^-_{\beta,\delta} 
\frac{a}{b} \beta \delta$, respectively, and combining \eqref{eq:cru0}-\eqref{eq:cru00}, we 
finally get the following relation, which holds for all $\beta, \delta \in [0,\epsilon_0)$, $h\in\R$ 
and $0 < a < b < \infty$:
\begin{equation} 
\label{eq:cru00+}
\f^{[a,b]} \big( \beta, h + C^-_{\beta,\delta} \tfrac{a}{b} \beta \delta ; 0 \big) 
\leq \f^{[a,b]}(\beta, h| \delta) 
\leq \f^{[a,b]} \big( \beta, h + C^+_{\beta,\delta} \tfrac{b}{a} \beta \delta ; 0 \big).
\end{equation}
Next, fix any $x > 0$ and $\eta > 0$. If $a_n \uparrow x$ and $b_n \downarrow x$, then
$a_n/b_n \geq 1-\eta$ and $b_n/a_n \leq 1+\eta$ for large $n$. Since $h \mapsto \f^{[a,b]}
(\beta, h;\delta)$ is non-decreasing, by \eqref{eq:parth} and \eqref{eq:s0}, 
for $n$ large enough we have
\begin{equation}
\f^{[a_n,b_n]} \big( \beta, h + C^-_{\beta,\delta} (1-\eta) \beta \delta ; 0 \big) 
\leq \f^{[a_n,b_n]}(\beta, h; \delta) 
\leq \f^{[a_n,b_n]} \big( \beta, h + C^+_{\beta,\delta} (1+\eta) \beta \delta ; 0 \big).
\end{equation}
Recalling \eqref{eq:fx} and \eqref{eq:fbetah}, we can let $n\to\infty$ to get that, for every 
$x > 0$,
\begin{equation}
\f^{\{x\}} \big( \beta, h + C^-_{\beta,\delta} (1-\eta) \beta \delta ; 0 \big) 
\leq \f^{\{x\}}(\beta, h; \delta) 
\leq \f^{\{x\}} \big( \beta, h + C^+_{\beta,\delta} (1+\eta) \beta \delta ; 0 \big).
\end{equation}
This relation also holds for $x=0$ because $\f^{\{0\}}(\beta, h; \delta)$ is a constant, as we 
showed in \eqref{eq:f0}. Taking the supremum over $x \in [0, s_0]$, we have 
shown that, for all $\beta, \delta \in [0,\epsilon_0)$ and $h\in\R$,
\begin{equation}
\f \big( \beta, h + C^-_{\beta,\delta} (1-\eta) \beta \delta ; 0 \big) 
\leq \f(\beta, h; \delta) 
\leq \f \big( \beta, h + C^+_{\beta,\delta} (1+\eta) \beta \delta ; 0 \big).
\end{equation}
Since $h \mapsto \f^{[a,b]}(\beta, h;\delta)$ is convex and finite, and hence continuous, we 
can let $\eta \downarrow 0$ to obtain \eqref{eq:compdeltah} for $\delta \in [0,\epsilon_0)$.

\medskip\noindent
\textbf{8.}
The case $\delta \in (-\epsilon_0, 0]$ is analogous. The inequality in \eqref{eq:cru0} is replaced by
\begin{gather} 
\label{eq:cru0neg}
C^-_{\beta,\delta} \, \beta \, a \, (-\delta) 
\leq \f^{[a,b]}(\beta, h; 0) \,-\, \f^{[a,b]}(\beta, h; \delta)
\leq C^+_{\beta,\delta} \, \beta \, b \, (-\delta),
\end{gather}
while \eqref{eq:cru00} for $\xi \leq 0$ becomes
\begin{equation} 
\label{eq:cru00neg}
a \, (-\xi)  \leq
\f^{[a,b]}(\beta, h ; 0) - \f^{[a,b]}(\beta, h + \xi ; 0) \leq b \, (-\xi).	
\end{equation}
Choosing $\xi = C^+_{\beta,\delta} \frac{b}{a} \beta \delta$ and $\xi = C^-_{\beta,\delta} 
\frac{a}{b} \beta \delta$, respectively, we get
\begin{equation} 
\label{eq:cru00+neg}
\f^{[a,b]} \big( \beta, h + C^+_{\beta,\delta} \tfrac{b}{a} \beta \delta ; 0 \big) 
\leq \f^{[a,b]}(\beta, h; \delta) 
\leq \f^{[a,b]} \big( \beta, h + C^-_{\beta,\delta} \tfrac{a}{b} \beta \delta ; 0 \big).
\end{equation}
It remains to let $a \uparrow x$, $b \downarrow x$, followed by taking the supremum
over $x \in [0,s_0]$.

\medskip\noindent
\textbf{9.}
Finally, by \eqref{eq:C+-}, we have $0 < C^-_{\beta,\delta} \le C^+_{\beta,\delta} 
< \infty$ for all $\beta \in [0,\epsilon_0)$ and $\delta \in (-\epsilon_0, \epsilon_0)$. By 
dominated convergence, $(\beta,\delta) \mapsto c^\pm_{\beta,\delta}$ are continuous 
on $[0,\epsilon_0) \times (-\epsilon_0,\epsilon_0)$, and hence also $(\beta,\delta) \mapsto 
C^\pm_{\beta,\delta}$ is continuous. 
Since $C^\pm_{0,0} = \bbvar(\omega_1) = 1$,
the proof is complete.
\qed

\smallskip

\section{Smoothing with respect to a shift: proof of Theorem~\ref{th:smoshift}}
\label{sec:smoothshift}

Equations \eqref{eq:ZNgen} and \eqref{eq:s0} imply that $h \mapsto \f(\beta,h;\delta)$ 
is non-decreasing. Since $\f(\beta,h;\delta) \geq 0$ under Assumption~\ref{ass:model},
by \eqref{eq:polylb}, if $\f(\bar\beta, \bar h; 0) = 0$, then $\f(\bar\beta, \bar h + t; 0) = 0$ 
for all $t \leq 0$, and \eqref{eq:compdeltah2} is trivially satisfied. Henceforth we assume
$t > 0$.

Recalling the statement of Theorem~\ref{th:compdeltah}, we set $F_\beta(\delta) 
:= C^-_{\beta,\delta} \, \delta$. This is a continuous and strictly increasing function 
of $\delta$, with $F_\beta(0) = 0$, and hence it maps the open interval $(0,\epsilon_0)$
into $(0,\epsilon'_0)$, for some $\epsilon'_0 > 0$. Applying the first inequality in 
\eqref{eq:compdeltah} for $t \in (0, \bar\beta\epsilon'_0)$, we can write
\begin{equation}
\f(\bar\beta, \bar h + t; 0) 
= \f\big(\bar\beta, \bar h + \bar\beta F_{\bar\beta}
(F_{\bar\beta}^{-1}(\tfrac{t}{\bar\beta})); 0\big) 
\leq \f\big(\bar\beta,\bar h; F_{\bar \beta}^{-1}(\tfrac{t}{\bar\beta}) \big).
\end{equation}
Applying \eqref{eq:smodelta}, we obtain
\begin{equation}
\f(\bar\beta, \bar h + t; 0) 
\leq \frac{\gamma}{2 \bar\beta^2} \, A_{\bar\beta,\frac{t}{\bar\beta}} \, t^2,
\end{equation}
where
\begin{equation}
A_{\beta,\delta} := B_{F_{\beta}^{-1}(\delta)} \, 
\bigg(\frac{F_{\beta}^{-1}(\delta)}{\delta} \bigg)^2.
\end{equation}
It follows from \eqref{eq:propCpm} that
$\lim_{(\beta,\delta) \to (0,0)} (F_{\beta}^{-1}(\delta)/\delta) = 1$.
Since $\lim_{\delta\to 0} B_\delta = 1$, we obtain $\lim_{(\beta,\delta) \to (0,0)} A_{\beta,\delta} 
= 1$.
\qed

%
%
%
%


\end{document}